\documentclass{amsart}% option [reqno] puts equat numb to the right

\usepackage[mathscr]{eucal}% So-called Euler fonts, allows \mathscr
\usepackage{amssymb}% allows special arrows like >-> e.g.
\usepackage[usenames,dvipsnames]{color}
\usepackage{amsthm}% allows theorem* , e.g.
\usepackage{bbold}% allows \unit \mathbb{1}
\usepackage{enumerate}% allows easy change of label
\usepackage{array}% allows array
\usepackage{amsmath}% allows pmatrix

\usepackage{hyperref}

\hypersetup{colorlinks=true,citecolor=Brown,linkcolor=Brown,urlcolor=Brown}% quote away (or =false) to generate ps%
\numberwithin{equation}{section}% makes equat numb contain the section
\usepackage[all]{xy}
\newdir{ >}{{}*!/-10pt/\dir{>}}

\swapnumbers %pour que le numéro apparaisse devant le théorème

\newtheorem{Thm}[equation]{Theorem}
\newtheorem*{Thm*}{Theorem}
\newtheorem{Prop}[equation]{Proposition}
\newtheorem{Lem}[equation]{Lemma}
\newtheorem{Cor}[equation]{Corollary}
\theoremstyle{remark}
\newtheorem{Def}[equation]{Definition}
\newtheorem{Exa}[equation]{Example}

\newtheorem{Rem}[equation]{Remark}
\newcommand{\nc}{\newcommand}
\nc{\dmo}{\DeclareMathOperator}

\dmo{\Abelem}{Abelem}
\dmo{\Add}{Add}
\dmo{\Free}{Free}
\dmo{\Id}{Id}
\dmo{\Loc}{Loc}
\dmo{\rmR}{R}
\dmo{\Thick}{Thick}
\dmo{\chara}{char}%
\dmo{\Coh}{Coh}
\dmo{\coker}{coker}
\dmo{\colim}{colim}
\dmo{\cone}{cone}
\dmo{\Der}{D}% ground notation for derived categories
\dmo{\End}{End}
\dmo{\Ext}{Ext}
\dmo{\rmH}{H}
\dmo{\Hom}{Hom}
\dmo{\id}{id}
\dmo{\Img}{Im}
\dmo{\incl}{incl}
\dmo{\inj}{inj}
\dmo{\Ker}{Ker}
\dmo{\Komp}{K}% ground notation for htpy categories
\dmo{\LL}{L}
\dmo{\locdeg}{locdeg}
\dmo{\Mod}{Mod}% sheaves of modules
\dmo{\modname}{mod}%
\dmo{\Mor}{Mor}%
\dmo{\Obj}{Obj}
\dmo{\Or}{Or}
\dmo{\pr}{pr}
\dmo{\Proj}{Proj}
\dmo{\proj}{proj}% projective modules
\dmo{\rank}{rank}
\dmo{\Res}{Res}
\dmo{\Rname}{R}
\dmo{\SH}{SH}% ground name for cat of spectra
\dmo{\smallb}{b}% ground exponent for ``bounded''
\dmo{\smallperf}{perf}% ground exponent for ``perfect''
\dmo{\Spc}{Spc}
\dmo{\Spec}{Spec}
\dmo{\stab}{stab}% stable category of fin. gen. mod.
\dmo{\supp}{supp}
\dmo{\TTR}{TTR}% tt-rings
\nc{\DbG}{\Db(\kk G\mmod)}%most used
\nc{\uA}{\underline{A}}
\nc{\doublequot}[3]{#1\backslash #2/#3}% double cosets
\nc{\HGK}{\doublequot HGK}% most used
\nc{\quadtext}[1]{\quad\textrm{#1}\quad}
\nc{\PZG}{\cat C_{\bbZ}(\bbZ G)}% G-lattices
\nc{\TTRK}{\TTR(\cat K)}% most used
\nc{\Gsets}{G\mathsf{-sets}}
\nc{\AddK}{\Add^{\Sigma}(\cat K)}
\nc{\adjto}{\rightleftarrows}
\nc{\AK}{A\MModcat{K}}% most used
\nc{\BK}{B\MModcat{K}}% most used
\nc{\BL}{B\MModcat{L}}% most used
\nc{\bbA}{\mathbb{A}}
\nc{\bbB}{\mathbb{B}}
\nc{\bbC}{\mathbb{C}}
\nc{\bbH}{\mathbb{H}}
\nc{\bbN}{\mathbb{N}}
\nc{\bbQ}{\mathbb{Q}}
\nc{\bbR}{\mathbb{R}}
\nc{\bbZ}{\mathbb{Z}}
\nc{\cat}[1]{\mathscr{#1}}%or: \nc{\cat}[1]{\mathcal{#1}}
\nc{\Displ}{\displaystyle}
\nc{\ie}{{i.e.}\ }
\nc{\into}{\mathop{\rightarrowtail}}
\nc{\inv}{^{-1}}
\nc{\isoto}{\buildrel \sim\over\to}
\nc{\isotoo}{\mathop{\buildrel \sim\over\too}}
\nc{\kk}{\Bbbk}
\nc{\onto}{\mathop{\twoheadrightarrow}}
\nc{\too}{\mathop{\longrightarrow}\limits}
\nc{\xytriangle}[7]{\xymatrix@C=#7em{#1\ar[r]^-{\Displ #4} & #2 \ar[r]^-{\Displ #5}&#3\ar[r]^-{\Displ #6}&T #1}}
\nc{\ababs}{{\sl ab absurdo}}
\nc{\adh}[1]{\overline{#1}}% adherence
\nc{\adhpt}[1]{\adh{\{#1\}}}% adherence of a pt
\nc{\ala}{{\sl \`a la}\ }
\nc{\calO}{\mathcal{O}}
\nc{\cV}{\mathcal{V}}% for support varieties
\nc{\Db}{\Der^{\smallb}}% derived bounded category
\nc{\Dperf}{\Der^{\smallperf}}% derived category of perfect compl
\nc{\eg}{{\sl e.\,g.}}
\nc{\env}{{^{\operatorname{e}}}}
\nc{\FFree}{\,\text{--}\Free}%
\nc{\FFreecat}[1]{\FFree_{\cat #1}}
\nc{\FK}{\mathcal{F}(\cat K)}% Freyd category
\nc{\gm}{\mathfrak{m}}% prime m
\nc{\Homcat}[1]{\Hom_{\cat #1}}
\nc{\hook}{\hookrightarrow}
\nc{\Idcat}[1]{\Id_{\cat{#1}}}
\nc{\ideal}[1]{\langle #1\rangle}
\nc{\Kb}{\Komp^{\smallb}}% htpy bounded category
\nc{\Kcat}[1]{#1\MModcat{K}}% modules over #1
\nc{\KP}{\cat{K}_{\cat P}}% local cat K at P
\nc{\loccit}{{\sl loc.\ cit.}}
\nc{\Lotimes}{\otimes^{\LL}}
\nc{\Mid}{\,\big|\,}
\nc{\MMod}{\,\text{-}\Mod}%
\nc{\MModcat}[1]{\MMod_{\cat #1}}%
\nc{\mmod}{\,\text{--}\modname}%
\nc{\mmodb}{\mmod^\sbull}%
\nc{\op}{{^{\operatorname{op}}}}
\nc{\oto}[1]{\overset{#1}{\,\to\,}}
\nc{\otoo}[1]{\overset{#1}{\,\too\,}}
\nc{\ourfrac}[2]{\genfrac{}{}{0pt}{}{\scriptstyle #1}{\scriptstyle #2}}
\nc{\ouriff}{\Leftrightarrow}
\nc{\oursetminus}{\!\smallsetminus\!}
\nc{\potimes}[1]{^{\otimes #1}}% tensor power
\nc{\pproj}{\,\text{--}\proj}%
\nc{\ptimes}[1]{^{\times #1}}% product power
\nc{\ppp}[1]{^{[#1]}}%
\nc{\pushout}{\textrm{\rm p.o.}}
\nc{\qp}{q_{_{\scriptstyle \cat P}}\!}%
\nc{\Rcat}[1]{\Rname_{\cat #1}^\sbull}% the graded endom ring of cat #1
\nc{\RK}{\Rcat{K}}
\nc{\sbull}{{\scriptscriptstyle\bullet}}%\mathbf{\cdot}}%{}}
\nc{\SET}[2]{\big\{\,#1\Mid#2\,\big\}}
\nc{\SHfin}{\SH^{\text{\rm fin}}}% stab. hom. cat of finite spectra
\nc{\smallmatrice}[1]{\left(\begin{smallmatrix} #1 \end{smallmatrix}\right)}
\nc{\Spccat}[1]{\Spc(\cat #1)}
\nc{\SpcAK}{\Spc(A\MModcat{K})}% most used
\nc{\SpcK}{\Spc(\cat K)}% most used
\nc{\SpcL}{\Spc(\cat L)}% most used
\nc{\suppcat}[1]{\supp(\cat #1)}
\nc{\then}{\Rightarrow}
\nc{\tideal}[1]{\ideal{#1}}%{\langle #1\rangle_{\!\s...style\otimes}}
\nc{\unit}{\mathbb{1}}% unit for \otimes
\nc{\unitcat}[1]{\unit_{\cat #1}}% above for the cat #1
\begin{document}

%------------------------------------------------------------------------------

\title{Splitting tower and degree of tt-rings}
\author{Paul Balmer}
\date{2013 December 2}

\address{Paul Balmer, Mathematics Department, UCLA, Los Angeles, CA 90095-1555, USA}
\email{balmer@math.ucla.edu}
\urladdr{http://www.math.ucla.edu/$\sim$balmer}

\begin{abstract}
After constructing a splitting tower for separable commutative ring objects in tensor-triangulated categories, we define and study their degree.
\end{abstract}

\subjclass{18E30, 13B40, 20J05, 55U35}
\keywords{Separable, tensor triangulated category, degree}

\thanks{Research supported by NSF grant~DMS-1303073.}

\maketitle

%------------------------------------------------------------------------------

\section*{Introduction}
\medskip
%------------------------------------------------------------------------------

Let $\cat K$ be a tensor-triangulated category (tt-category, for short). Denote its tensor by $\otimes:\cat K\times\cat K\too\cat K$ and its $\otimes$-unit by~$\unit$. Let $A$ be a ring object in~$\cat K$, that is, an associative monoid $\mu:A\otimes A\to A$ with unit $\eta:\unit\to A$. We want to study the \emph{degree} of such a ring object under the assumption that $A$ is what we call a \emph{tt-ring}, \ie is commutative and separable. We focuss on tt-rings because their Eilenberg-Moore category $\AK$ of $A$-modules in~$\cat K$ remains a tt-category and extension-of-scalars $F_A:\cat K\too\AK$ is a tt-functor (a fact which also explains the terminology\,: tt-rings preserve tt-categories). See Section~\ref{se:prelim}.

In practice, tt-rings appear in commutative algebra as finite \'etale algebras and in representation theory of finite groups as the amusing algebras $A=\kk(G/H)$ associated to subgroups~$H<G$; see~\cite{Balmer12bpp}. In the latter case, if $\cat K=\cat K(G)$ is the derived or the stable category of the group~$G$ over a field~$\kk$, then $\AK$ is nothing but the corresponding category $\cat K(H)$ for the subgroup. These two sources already provide an abundance of examples. Furthermore, the topological reader will find tt-rings among ring spectra, equivariant or not.

Let us contemplate the problem of defining a reasonable notion of degree, \ie an integer $\deg(A)$ measuring the size of the tt-ring $A$ in a general tt-category~$\cat K$. When working over a field~$\kk$, it is tempting to use $\dim_\kk(A)$. When $A$ is a projective separable $R$-algebra over a commutative ring~$R$ then its rank must be finite (\cite{DeMeyerIngraham71}) and provides a fine notion of degree for $A$ viewed in the tt-category of perfect complexes~$\Dperf(R)=\Kb(R\pproj)$. However, tt-geometry covers more than commutative algebra. Unorthodox separable algebras already emerge in representation theory for instance, as the above~$A=\kk(G/H)$. In $\DbG$, one can still forget the $G$-action and take dimension over~$\kk$ as a possible degree -- which of course yields the index $[G\!:\!H]$ in that example -- but one step further, in the stable category $\cat K=\stab(\kk G)$, dimension over~$\kk$ becomes a slippery notion whereas the tt-ring $A=\kk(G/H)$ remains equally important. Such questions become even harder in general stable homotopy categories (see~\cite{HoveyPalmieriStrickland97}) where there is simply no ground field~$\kk$ to deal with in the first place. So a good concept of degree in the broad generality of tt-geometry requires a new idea.

Our solution relies on the following Splitting Theorem~\ref{thm:split}, which echoes a classical property of usual separable rings (see~\cite{ChaseHarrisonRosenberg65}). Note that such a result is completely wrong for non-separable rings, already with $B=A[X]$ for instance.

\begin{Thm*}
Let $f:A\to B$ and $g:B\to A$ be homomorphisms of tt-rings such that $g\circ f=\id_A$. Then there exists a tt-ring
$C$ and a ring isomorphism $\smallmatrice{g\\\ast}:B\isoto A\times C$.
\end{Thm*}

Using this theorem, we construct (Definition~\ref{def:tower}) a tower of tt-rings and homomorphisms $A=:A\ppp{1}\to A\ppp{2}\to\cdots \to A\ppp{n}\to\cdots$ such that after extending scalars to $A\ppp{n}$ our $A$ splits as the product of $\unit\times\cdots\times\unit$ ($n$ times) with~$A\ppp{n+1}$. The \emph{degree} of~$A$ is defined to be the last $n$ such that $A\ppp{n}\neq0$ (Definition~\ref{def:deg}).

We prove a series of results which show that this concept of degree behaves according to intuition and provides a reasonable invariant. In basic examples, we recover expected values, like $[G\!:\!H]$ in the case of $\kk(G/H)$ in~$\DbG$. In the stable category however, $\deg(\kk(G/H))$ can be smaller than~$[G\!:\!H]$. In the extreme case of $H< G$ strongly $p$-embedded, we even get $\deg(\kk(G/H))=1$ in $\stab(\kk G)$; see Example~\ref{exa:nat-G/H}. We prove in Section~\ref{se:exas} that the degree is finite for every tt-ring in the derived category of perfect complexes over a scheme, in the bounded derived category of a finite dimensional cocommutative Hopf algebra and in the stable homotopy category of finite spectra.

It is an open question whether the degree must always be finite, at least locally.

Several aspects of this work extend to non-triangulated additive tensor categories. This is discussed in Remark~\ref{rem:non-t}.

In~\cite{Balmer13ppb}, our degree theory will be used to control the fibers of the map $\Spc(\AK)\to \SpcK$. We shall notably reason by induction on the degree, thanks to the following Theorem~\ref{thm:ind-deg}\,:
\begin{Thm*}
Let $A$ be a tt-ring of finite degree $d$ in~$\cat K$. Then in the tt-category of $A$-modules, we have an isomorphism of tt-rings $F_A(A)\simeq\unit\times A\ppp{2}$ where the tt-ring $A\ppp{2}$ has degree $d-1$ in~$\AK$.
\end{Thm*}

\medbreak

\noindent\textit{Convention}\,: All our tt-categories are essentially small and idempotent complete.

%------------------------------------------------------------------------------
\goodbreak
\section{The tt-category of $A$-modules}
\label{se:prelim}%
\medskip
%------------------------------------------------------------------------------

We quickly list standard properties of the Eilenberg-Moore category~$\AK$ of $A$-modules in~$\cat K$, see~\cite{EilenbergMoore65}, \cite[Chap.\,VI]{MacLane98} or~\cite{Balmer11}.

As $A$ is \emph{separable} (\ie $\mu:A\otimes A\to A$ has a section $\sigma$ as $A,A$-bimodules) the category $\AK$ admits a unique triangulation such that both extension-of-scalars $F_A:\cat K\too \AK$, $x\mapsto A\otimes x$, and its forgetful right adjoint $U_A:\AK\too\cat K$ are exact; see~\cite{Balmer11}. Also, $\AK$ is equivalent to the idempotent completion of the \emph{Kleisli category} $A\FFreecat{K}$ of free $A$-modules, see~\cite{Kleisli65}. Objects of $A\FFreecat{K}$ are the same as those of~$\cat K$, denoted $F_A(x)$ for every $x\in\cat K$, and morphisms $\Hom_{A}(F_A(x),F_A(y)):=\Homcat{K}(x,A\otimes y)$. Denote by $\bar f:F_A(x)\to F_A(y)$ the morphism in $A\FFreecat{K}$ corresponding to $f:x\to A\otimes y$ in~$\cat K$.

As our tt-ring $A$ is furthermore \emph{commutative}, there is a tensor structure $-\otimes_A-:\AK\ \times\ \AK\too\AK$ making $F_A:\cat K\too \AK$ a tt-functor. Indeed, one can define $\otimes_A$ on the Kleisli category by $F_A(x)\otimes_A F_A(y):=F_A(x\otimes y)$ and $\bar f\otimes_A \bar g=\overline{(\mu\otimes1\otimes1)(23)(f\otimes g)}$ if $f:x\to A\otimes x'$ and $g:y\to A\otimes y'$, as below\,:
$$
\xymatrix@C=3em{
x\otimes y \ar[r]^-{f\otimes g}
& A\otimes x'\otimes A\otimes y' \ar[r]^-{(23)}
& A\otimes A\otimes x'\otimes y' \ar[r]^-{\mu\otimes 1\otimes 1}
& A\otimes x'\otimes y'.
}
$$
Idempotent completion then yields $\otimes_A$ on $\AK$. One can also describe $\otimes_A$ on all modules directly. First only assume that $A$ is separable with a chosen $A,A$-bimodule section $\sigma:A\to A\otimes A$ of~$\mu$. Let $(x_1,\varrho_1)$ and $(x_2,\varrho_2)$ be right and left $A$-modules in~$\cat K$ respectively. Then the following endomorphism $v$ in~$\cat K$
\begin{equation}
\label{eq:v}%
\xymatrix@C=2.5em{
v:\,x_1\otimes x_2 \ar[r]^-{1\otimes \eta\otimes 1}
& x_1\otimes A\otimes x_2 \ar[r]^-{1\otimes \sigma\otimes 1}
& x_1\otimes A\otimes A\otimes x_2 \ar[r]^-{\varrho_1\otimes \varrho_2}
& x_1\otimes x_2 }
\end{equation}
is an idempotent\,: $v\circ v=v$. Hence one can define $x_1\otimes_A x_2:=\Img(v)$ as the corresponding direct summand of~$x_1\otimes x_2$. We obtain a split coequalizer in~$\cat K$\,:
$$
\xymatrix@C=4em{x_1\otimes A\otimes x_2\
\ar@<.2em>[r]^{\varrho_1\otimes1}
\ar@<-.2em>[r]_{1\otimes \varrho_2}
&\ x_1\otimes x_2\
 \ar@{->>}[r]^-{v}
&\ \Img(v)=x_1\otimes_A x_2
}
$$
as in the traditional definition of $\otimes_A$. When $A$ is commutative, left and right $A$-modules coincide and one induces an $A$-action on $x_1\otimes_Ax_2$ from the usual formula on $x_1\otimes x_2$. One verifies that this coincides with the tensor constructed above.

\begin{Prop}[Projection formula]
\label{prop:proj}%
Let $A$ be a tt-ring in~$\cat K$. For all $y\in \cat K$ and $x\in A\MModcat{K}$, we have a natural isomorphism $U_A(x\otimes_A F_A(y))\cong U_A(x)\otimes y$ in~$\cat K$.
\end{Prop}

\begin{proof}
By construction of~$\otimes_A$, it suffices to prove the existence of such an isomorphism for $x\in A\FFreecat{K}$, which is natural in~$x$ in that category (and in~$y$ too but that is easy). So, let $x=F_A(z)$ for some $z\in \cat K$. Then $U_A(x\otimes F_A(y))=U_A(F_A(z)\otimes_A F_A(y))=U_A(F_A(z\otimes y))=A\otimes (z\otimes y)\cong (A\otimes z) \otimes y=U_A(x)\otimes y$. This looks trivial but the point is that this isomorphism is natural with respect to morphisms $\bar f:x=F_A(z)\to F_A(z')=x'$ in~$A\FFreecat{K}$ for $f:z\to A\otimes z'$ in~$\cat K$ (not just natural in~$z$). This is now an easy verification.
\end{proof}

\begin{Rem}
\label{rem:sep-factor}%
For two ring objects $A$ and $B$, the ring object $A\times B$ is $A\oplus B$ with component-wise structure. The ring object $A\otimes B$ has multiplication $(\mu_1\otimes\mu_2)\, (23):(A\otimes B)\potimes{2}\too A\otimes B$ and obvious unit. The \emph{opposite} $A\op$ is~$A$ with $\mu\op=\mu\,(12)$. The \emph{enveloping ring} $A\env$ is $A\otimes A\op$. Left $A\env$-modules are just $A,A$-bimodules.

If $A$ and $B$ are separable, then so are $A\times B$, $A\otimes B$ and $A\op$. Conversely, if $A\times B$ is separable then so are $A$ and $B$ (restrict the section~``$\sigma$" to each factor).
\end{Rem}

\begin{Rem}
\label{rem:AB}%
Let $h:A\to B$ be a homomorphism of tt-rings in~$\cat K$ (\ie $h$ is compatible with multiplications and units). We also say that $B$ is an \emph{$A$-algebra} or a tt-ring \emph{over~$A$}. Then idempotent-complete the functor $F_h:A\FFreecat{K}\too B\FFreecat{K}$ defined on objects by $F_h(F_A(x))=F_B(x)$ and on morphisms by $F_h(\bar f)=\overline{(h\otimes 1)\circ f}$.
Alternatively, equip $B$ with a right $A$-module structure via~$h$ and define for every $A$-module $x\in\AK$, its extension $F_h(x)=B\otimes_Ax$ equipped with the left $B$-module structure on the~$B$ factor. Both define the same tt-functor $F_h:A\MModcat{K}\too B\MModcat{K}$ and the following diagram commutes up to isomorphism\,:
$$
\xymatrix@R=1.5em{
& \cat K \ar[ld]_-{F_A} \ar[rd]^-{F_B}
\\
A\MModcat{K} \ar[rr]^-{F_h\,\cong \,B\otimes_A-}
&& B\MModcat{K}\,.
}
$$
Furthermore, if $k:B\to C$ is another homomorphism then $F_{kh}\cong F_k\circ F_h$.
\end{Rem}

\begin{Rem}
\label{rem:alg}%
For $A$ a tt-ring in~$\cat K$, there is a one-to-one correspondence between
\begin{enumerate}[(i)]
\item $A$-algebras in~$\cat K$, \ie homomorphism $h:A\to B$ of tt-rings in~$\cat K$, and
\item tt-rings $\bar B$ in $\AK$.
\end{enumerate}
The correspondence is the obvious one\,: To every tt-ring $\bar B=(\bar B,\bar \mu,\bar \eta)$ in $\AK$, associate  $B:=U_A(\bar B)$ and $h:=U_A(\bar\eta)$. The ring structure on~$B$ is given by $B\otimes B=U_A(\bar B)\otimes U_A(\bar B)\overset{v}{\,\onto\,} U_A(\bar B\otimes_A \bar B)\oto{\bar \mu}U_A(\bar B)=B$ and $\eta_B:\unit\oto{\eta_A}A\oto{h} B$. Conversely, if $h:A\to B$ is a homomorphism, then one can use $h$ to equip $\bar B:=B$ with an $A$-module structure and verify that $\mu:B\otimes B\to B$ respects the idempotent $v$ of~\eqref{eq:v}, hence defines $\bar \mu:B\otimes_AB\to B$. Then $B$ is separable in~$\cat K$ (with section $\sigma$ of~$\mu$) if and only if $\bar B$ is separable in~$\AK$ (with section $v\,\sigma$ of~$\bar \mu$).

We tacitly use this dictionary below. If we need to distinguish the $A$-algebra~$B$ in~$\cat K$ from the tt-ring $\bar B$ in~$\AK$, we shall write $U_A(\bar B)$ for the former.

Under this correspondence, if $\cat L:=\AK$, there is an equivalence $B\MModcat{K}\cong\bar B\MModcat{L}$ such that the following diagram commutes up to isomorphism\,:
$$
\xymatrix{\cat K \ar[r]^-{F_A} \ar[d]_-{F_B}
& \cat L=\AK \ar[d]^-{F_{\bar B}}
\\
B\MModcat{K}\ar@{-}[r]^-{\simeq}
& \bar B\MModcat{L}\,.
}
$$
On Kleisli categories, it maps $F_B(x)$ to $F_{\bar B}(F_A(x))$ for every $x\in \cat K$ and follows the sequence of isomorphisms $\Hom_B(F_B(x),F_B(y))\cong \Homcat{K}(x,B\otimes y)\cong\Hom_A(F_A(x),B\otimes y)\cong\Hom_A(F_A(x),\bar B\otimes_A F_A(y))\cong\Hom_{\bar B}(F_{\bar B}F_A(x),F_{\bar B}F_A(y))$ on morphisms. Idempotent completion does the rest.
\end{Rem}

\begin{Rem}
\label{rem:F-Mod}%
Let $F:\cat K\to \cat L$ be a tt-functor. Let $A$ be a tt-ring in~$\cat K$ and let $B:=F(A)$ its image in~$\cat L$. Then $B$ is also a tt-ring and there exists a tt-functor $\bar F:A\MModcat{K}\too B\MModcat{L}$
\begin{equation}
\label{eq:barF}%
\vcenter{\xymatrix{\cat K \ar[r]^-{F} \ar@<-.2em>[d]_-{F_A}
& \cat {L} \ar@<-.2em>[d]_-{F_B}
\\
A\MModcat{K} \ar@<-.2em>[u]_-{U_A} \ar[r]^-{\bar F}
& B\MModcat{L} \ar@<-.2em>[u]_-{U_B}
}}
\end{equation}
such $\bar FF_A=F_BF$ and $U_B\bar F=F\,U_A$. Explicitly, for every $A$-module $(x,\varrho)$, we have $\bar F(x,\varrho)=(F(x),F(\varrho))$ where $B\otimes F(x)=F(A)\otimes F(x)\cong F(A\otimes x)\otoo{F(\varrho)} F(x)$. On morphisms, $\bar F(f)=F(f)$. The ``Kleislian" description of~$\bar F$ is equally easy.
\end{Rem}

%------------------------------------------------------------------------------
\goodbreak
\section{Splitting theorems}
\label{se:split}%
\medskip
%------------------------------------------------------------------------------

We will iteratively use the following splitting result\,:

\begin{Thm}
\label{thm:split}%
Let $f:A\to B$ and $g:B\to A$ be homomorphisms of tt-rings in $\cat K$ such that $g\circ f=\id_A$. Then there exists a tt-ring~$C$ and a ring isomorphism $h:B\isoto A\times C$ such that $\pr_1\, h=g$. Consequently, $C$ becomes an $A$-algebra, via $\pr_2\,h\,f$. Moreover, if $C'$ is another $A$-algebra and $h':B\isoto A\times C'$ is another $A$-algebra isomorphism such that $\pr_1\,h=g$, then there exists an isomorphism of $A$-algebras $\ell:C\isoto C'$ such that $h'=(1\times\ell)\,h$.
\end{Thm}

We start with a couple of additive lemmas.

\begin{Lem}
\label{lem:bimod}%
Let $B$ be a ring object, $B_1$ and $B_2$ two $B,B$-bimodules and $h:B\isoto B_1\oplus B_2$ an isomorphism of $B,B$-bimodules. Then $B_1$ and $B_2$ admit unique structures of ring objects such that $h$ becomes a ring isomorphism $B\isoto B_1\times B_2$.
\end{Lem}

\begin{proof}
Write the given isomorphisms $h=\smallmatrice{h_1\\h_2}:B \isotoo B_1\oplus B_2$ and $h\inv=\smallmatrice{k_1 & k_2}:B_1\oplus B_2\isotoo B$. If $h$ is to be a ring isomorphism, we must have for $i=1,2$ that the multiplication $\mu_i:B_i\otimes B_i\to B_i$ is given by $\mu_i=h_i\,\mu\,(k_i\otimes k_i)$ and the unit $\eta_i:\unit\to B_i$ by $\eta_i=h_i\, \eta$. Hence uniqueness. Conversely, let us see that these formulas provide the wanted ring structures. Let $\rho:B\otimes B_2\too B_2$ be the left $B$-action on~$B_2$. By left $B$-linearity of $k_2:B_2\to B$, we have $\mu\,(1\otimes k_2)=k_2\, \rho:B\otimes B_2\to B$. Note that $h\,h\inv =\smallmatrice{1&0\\0&1}$ implies $h_ik_j=0$ when $i\neq j$. Therefore $h_1\,\mu\,(1\otimes k_2)=h_1\, k_2\, \rho=0$. Similarly, $B,B$-linearity of~$k_1$ and $k_2$ gives $h_i\,\mu\,(1\otimes k_j)=0$ and $h_i\,\mu\,(k_j\otimes 1)=0$ when $i\neq j$. So the bottom square of the following diagram commutes, in which the top one commutes by definition\,:
$$
\vcenter{\xymatrix@C=6em@R=2em{
(B_1\!\oplus\! B_2)\otimes (B_1\!\oplus\! B_2)\ar[d]^-{\cong} \ar[r]^-{\mu_{B_1\times B_2}}
 \ar@/_11em/[dd]_(.2){h\inv\otimes h\inv}^(.2){\simeq}
& B_1\oplus B_2 \ar@{=}[d]
\\
(B_1\!\otimes\! B_1)\oplus(B_1\!\otimes\! B_2)\oplus(B_2\!\otimes\! B_1)\oplus(B_2\!\otimes\! B_2)
 \ar[r]^-{\smallmatrice{\mu_1&0&0&0\\0&0&0&\mu_2}}
 \ar[d]^-{\smallmatrice{k_1\otimes k_1 & k_1\otimes k_2 &k_2\otimes k_1 &k_2\otimes k_2}}
& B_1\oplus B_2
\\
B\otimes B \ar[r]^-{\mu}
& B \ar[u]^-{h=\smallmatrice{h_1\\h_2}}_-{\simeq}
}}
$$
Hence $h:B\isoto B_1\oplus B_2$ is an isomorphism of objects-equipped-with-multiplications. Since $B$ is associative and unital, $B_1$ and $B_2$ must have the same properties.
\end{proof}

\begin{Lem}
\label{lem:triang}%
Let $C$ and $C'$ be ring objects and $\smallmatrice{1&0\\s&\ell}:\unit\times C\isoto \unit\times C'$ a ring isomorphism. Then $s=0$ and $\ell$ is a ring isomorphism.
\end{Lem}

\begin{proof}
Let us denote by $(C,\mu,\eta)$ and $(C',\mu',\eta')$ the structures. Clearly $\ell$ is an isomorphism of objects. From the fact that $\smallmatrice{1&0\\s&\ell}$ preserves the structures it follows that $\eta'=s+\ell\eta$ and that $\smallmatrice{1&0&0&0\\0&0&0&\mu'}\big(\smallmatrice{1&0\\s&\ell}\otimes\smallmatrice{1&0\\s&\ell}\big)=\smallmatrice{1&0\\s&\ell}\smallmatrice{1&0&0&0\\0&0&0&\mu}$, giving in particular $\mu'(s\otimes \ell)=0$ and $\mu'\,(\ell\otimes \ell)=\ell\mu$. Composing the former with $(1\otimes \ell\inv\eta'):\unit\otimes \unit\to \unit\otimes C$, we get $0=\mu'(s\otimes \ell)(1\otimes \ell\inv\eta')=\mu'(s\otimes \eta')=\mu'(1_{C'}\otimes \eta')s=s$ and therefore $\ell\eta=\eta'$. Hence $\ell$ preserves multiplication and unit.
\end{proof}

\begin{proof}[Proof of Theorem~\ref{thm:split}]
Via the morphism~$g:B\to A$, we can equip $A$ with a structure of $B,B$-bimodule, so that $g$ becomes $B\env$-linear. Since $B\env=B\otimes B\op$ is separable, the category $B\env\MModcat{K}$ of $B,B$-bimodules is triangulated in such a way that $U_{B\env}:B\env\MModcat{K}\too\cat K$ is exact. Choose a distinguished triangle over~$g$ in $B\env\MModcat{K}$ say $C\to B\overset{g}\to A\overset{z}\to \Sigma C$. Forgetting the $B\env$-action, since $g$ is split by~$f$ in~$\cat K$, we see that $U_{B\env}(z)=0$. Since $U_{B\env}$ is faithful, $z$ is also zero in the triangulated category $B\env\MModcat{K}$, which in turn yields an isomorphism $h:B\isoto A\oplus C$ of $B,B$-bimodules such that $\pr_1\, h=g$. By Lemma~\ref{lem:bimod}, $A$ and $C$ can be equipped with ring structures so that $h$ is a ring isomorphism. We are left to verify that this new ring structure on~$A$ is indeed the original one. This follows from the fact that $g:B\to A$ is a split epimorphism which is a homomorphism from $B$ to $A$ with both structures (the original one by hypothesis and the new one because $h=\smallmatrice{g\\\ast}$ is a homomorphism). Note that $C$ is separable by Remark~\ref{rem:sep-factor}. Finally, for uniqueness of~$C$ as $A$-algebra, with the notation of the statement, we obtain an isomorphism $k:=h'\circ h\inv:\unit\times C\isoto \unit\times C'$ in~$\AK$ such that $\pr_1\, k=\pr_1$ which means that $k$ has the form $\smallmatrice{1&0\\s&\ell}$ and we conclude by Lemma~\ref{lem:triang}.
\end{proof}

\begin{Thm}
\label{thm:A'}%
Let $A$ be a tt-ring in~$\cat K$. Then there exists a ring isomorphism $h:A\otimes A\isoto A\times A'$ for some tt-ring~$A'$ in such a way that $\pr_1\, h=\mu$. Moreover, the $A$-algebra $A'$ is unique up to isomorphism with this property.
\end{Thm}

\begin{proof}
Apply Theorem~\ref{thm:split} to the tt-ring $B=A\otimes A$ with $g=\mu:A\otimes A\to A$ and $f=1_A\otimes\eta=F_A(\eta):A\to A\otimes A$.
\end{proof}

\begin{Rem}
\label{rem:A'}%
From the isomorphism $\smallmatrice{\mu\\\ast}:\,A\otimes A\isoto A\oplus A'$, we see that $A'\simeq\Sigma\inv\!\cone(\mu)\simeq\cone(1_A\otimes\eta)\simeq A\otimes\cone(\eta)$ in~$\cat K$. Also $\supp(A')\subseteq\supp(A)$.
\end{Rem}

%------------------------------------------------------------------------------
\goodbreak
\section{Splitting tower and degree}
\label{se:deg}%
\medskip
%------------------------------------------------------------------------------

%
\begin{Def}
\label{def:tower}%
We define \emph{the splitting tower} of a tt-ring $A$
$$
A\ppp{0}\to A\ppp{1}\to A\ppp{2}\to\cdots \to A\ppp{n}\to A\ppp{n+1}\to\cdots
$$
as follows\,: We start with $A\ppp{0}=\unit$, $A\ppp{1}=A$ and $\eta:A\ppp{0}\to A\ppp{1}$. Then for $n\geq 1$ we define $A\ppp{n+1}=(A\ppp{n})'$ in the notation of Theorem~\ref{thm:A'} applied to the tt-ring $A\ppp{n}$ in the tt-category $A\ppp{n-1}\MModcat{K}$ (see Remark~\ref{rem:alg}). Equivalently, $A\ppp{n+1}$ is characterized as $A\ppp{n}$-algebra by the existence of an isomorphism of $A\ppp{n}$-algebras
\begin{equation}
\label{eq:An}%
h:A\ppp{n}\otimes_{A\ppp{n-1}}A\ppp{n}\isoto A\ppp{n}\times A\ppp{n+1}
\end{equation}
such that $\pr_1\, h=\mu$. This tower $\{A\ppp{n}\}_{n\geq0}$ is well-defined up to isomorphism.
\end{Def}

\begin{Rem}
\label{rem:n->n+1}%
By Remark~\ref{rem:A'}, $\supp(A\ppp{n+1})\subseteq \supp(A\ppp{n})$, and if $A\ppp{n}=0$ for some~$n$ then $A\ppp{m}=0$ for all $m\geq n$. Also, by construction, if we consider $A\ppp{n}$ as a tt-ring in $A\ppp{n-1}\MModcat{K}$, we have $(A\ppp{n})\ppp{m}\cong A\ppp{n+m-1}$ for all $m\geq 1$.
\end{Rem}

\begin{Def}
\label{def:deg}%
We say that $A$ \emph{has finite degree~$d$} if $A\ppp{d}\neq0$ and $A\ppp{d+1}=0$. In that case, we write $\deg(A)=d$ or $\deg_{\cat K}(A)=d$ if we need to stress the category. If $A\ppp{n}\neq0$ for all~$n\geq0$, we say that $A$ has \emph{infinite degree}.
\end{Def}

\begin{Exa}
For $A=\unit\times \unit$, we have $A\otimes A\simeq A\times A$. Hence $A\ppp{2}=A=A\ppp{1}$. If one was to compute $A\ppp{2}\otimes A\ppp{2}$ one would get $A\ppp{2}\times A$ again and misreading Definition~\ref{def:deg} could lead to the false impression that $A\ppp{3}$ is $A$ again and that all $A\ppp{n}$ are equal. This is not the way to compute $A\ppp{3}$\,! One needs to compute $A\ppp{2}\otimes_{A\ppp{1}}A\ppp{2}=A\otimes_AA=A=A\ppp{2}\times 0$ and therefore $A\ppp{3}=0$. So, the tt-ring $\unit\times \unit$ has degree~2. In~\eqref{eq:An}, it is important to perform the tensor \textbf{over}~$\mathbf{A\ppp{n-1}}$.
\end{Exa}

An immediate gain of having a numerical invariant like the degree is the possibility of making proofs by induction. This will be applied in~\cite{Balmer13ppb} using the Splitting Theorem~\ref{thm:A'}, under the following form\,:
\begin{Thm}
\label{thm:ind-deg}%
Let $A$ be a tt-ring of finite degree~$d$ in a tt-category~$\cat K$. Then we have a ring isomorphism $F_A(A)\simeq\unit_A\times A\ppp{2}$ and $\deg(A\ppp{2})=d-1$ in~$\AK$.
\end{Thm}

\begin{proof}
Since $A\ppp{2}=A'$, this is simply Theorem~\ref{thm:A'} with $A$-algebras replaced by tt-rings in~$\AK$ (see Remark~\ref{rem:alg} if necessary), together with the observation that $(A\ppp{2})\ppp{n}=A\ppp{n+1}$ for all~$n\geq 1$, which gives $\deg_{\AK}(A\ppp{2})=\deg_{\cat K}(A)-1$.
\end{proof}

Before showing in Section~\ref{se:exas} that many tt-rings have finite degree, let us build our understanding of this $\deg(A)\in\bbN\cup\{\infty\}$, starting with functorial properties.
\goodbreak

\begin{Thm}
\label{thm:nat}%
Let $A$ be a tt-ring in~$\cat K$.
\begin{enumerate}[\rm(a)]
\item
\label{it:nat-1}%
Let $F:\cat K\to \cat L$ be a tt-functor. Then for every~$n\geq 0$, we have $F(A)\ppp{n}\simeq F(A\ppp{n})$ as tt-rings. In particular, $\deg(F(A))\leq \deg(A)$.
\smallbreak
\item
\label{it:nat-2}%
Let $F:\cat K\to \cat L$ be a tt-functor. Suppose that $F$ is ``\emph{weakly conservative on~$\supp(A)$}", \ie for $x\in\cat K_{\supp(A)}$ if $F(x)=0$ then $x\potimes{m}=0$ for $m\geq0$; for instance, if $F$ is just conservative. Then $\deg(F(A))=\deg(A)$.
\smallbreak
\item
\label{it:nat-3}%
Suppose that $B\in\cat K$ is a tt-ring such that $\supp(B)\supseteq\supp(A)$; for instance, if $\supp(B)=\SpcK$. Then $\deg(A)$ is equal to the degree of~$F_B(A)$ in~$B\MModcat{K}$.
\smallbreak
\item
\label{it:nat-4}%
Suppose $\cat K$ is local and that $B\in\cat K$ is a non-zero tt-ring. Then the degree of~$A$ in $\cat K$ is equal to the degree of~$F_B(A)$ in~$B\MModcat{K}$.
\end{enumerate}
\end{Thm}

\begin{proof}
To prove~\eqref{it:nat-1} by induction on~$n$, simply apply $F$ to~\eqref{eq:An}, which characterizes the splitting tower~$A\ppp{\ast}$. So, if $\deg(A)<\infty$ then $F(A)\ppp{\deg(A)+1}\simeq F(A\ppp{\deg(A)+1})=0$ and $\deg(F(A))\leq \deg(A)$. For~\eqref{it:nat-2}, recall that $A\ppp{n}\in\cat K_{\supp(A)}$ for every $n\geq1$ (Rem.\,\ref{rem:n->n+1}). As $\deg(F(A))=:d<\infty$ implies $F(A\ppp{d+1})\simeq F(A)\ppp{d+1}=0$ we get by weak-conservativity of~$F$ that $A\ppp{d+1}$ is $\otimes$-nilpotent, hence zero (every ring object is a direct summand of its $\otimes$-powers via the unit). This $A\ppp{d+1}=0$ means $\deg(A)\leq d=\deg(F(A))$, which finishes~\eqref{it:nat-2}. Then~\eqref{it:nat-3} follows since $\supp(B)\supseteq\supp(A)$ implies that $F_B:\cat K\to B\MModcat{K}$ is weakly conservative on~$\cat K_{\supp(A)}$. Indeed, if $\supp(x)\subseteq\supp(A)$ and $F_B(x)=0$ then $B\otimes x=U_B F_B(x)=0$ and $\varnothing=\supp(B\otimes x)=\supp(B)\cap \supp(x)=\supp(x)$, which implies that $x$ is $\otimes$-nilpotent. For~\eqref{it:nat-4}, recall that a tt-category is \emph{local} if $x\otimes y=0$ implies that $x$ or $y$ is $\otimes$-nilpotent. Hence for the non-zero tt-ring~$B$, the functor $F_B:\cat K\to B\MModcat{K}$ is weakly conservative on the whole of~$\cat K$ and we can apply~\eqref{it:nat-2}. \end{proof}

Let us now describe the local nature of the degree. Recall that for every prime $\cat P\in\SpcK$, the local category $\KP=(\cat K/\cat P)^\natural$ at~$\cat P$ is the idempotent completion of the Verdier quotient $\cat K/\cat P$, hence comes with a tt-functor $\qp:\cat K\onto \cat K/\cat P\hook\KP$.

\begin{Thm}
\label{thm:local}%
Let $A$ be a tt-ring in~$\cat K$. Suppose that $\qp(A)$ has finite degree in the local tt-category $\KP$ for every point $\cat P\in\SpcK$. Then $A$ has finite degree and
$$
\deg(A)=\max_{\cat P\in\SpcK}\deg(\qp(A))=\max_{\cat P\in\supp(A)}\deg(\qp(A))\,.
$$
\end{Thm}

\begin{proof}
There exists, for every $\cat P\in \SpcK$, an integer $n_{\cat P}\geq 1$ such that $\qp(A\ppp{n_{\cat P}})=(\qp(A))\ppp{n_{\cat P}}=0$. Hence $\cat P$ belongs to the open $\cat U(A\ppp{n_{\cat P}}):=\SpcK-\supp(A\ppp{n_{\cat P}})$. Putting all those open subsets together, we cover $\SpcK$. But the spectrum is always quasi-compact and $\cat U(A\ppp{n})\subseteq \cat U(A\ppp{n+1})$, hence there exists $n\geq0$ such that $\cat U(A\ppp{n})=\SpcK$. This means $A\ppp{n}=0$, \ie $d:=\deg(A)<\infty$. By Theorem~\ref{thm:nat}\,\eqref{it:nat-1} we have $d=\deg(A)\geq\max_{\cat P\in\SpcK}\deg(\qp(A))\geq\max_{\cat P\in\supp(A)}\deg(\qp(A))$. Since $A\ppp{d}\neq0$ there exists $\cat P\in\supp(A\ppp{d})\subseteq\supp(A)$ with $0\neq\qp(A\ppp{d})\simeq (\qp(A))\ppp{d}$ hence $\deg(\qp(A))\geq d=\deg(A)$, wrapping up all above inequalities into equalities.
\end{proof}

We now discuss the link between the degree and the trivial tt-ring~$\unit$.

\begin{Thm}
\label{thm:1}%
Let $A$ be a tt-ring in~$\cat K$. Suppose $\cat K\neq0$.
\begin{enumerate}[\rm(a)]
\smallbreak
\item
\label{it:1-0}%
For every $n\geq1$, we have $\deg(\unit\ptimes{n})=n$.
\smallbreak
\item
\label{it:1-1}%
For every $n\geq1$ we have $F_{A\ppp{n}}(A)\simeq \unit\ptimes{ n}\times A\ppp{n+1}$ as tt-rings in $A\ppp{n}\MModcat{K}$.
\smallbreak
\item
\label{it:1-2}%
If $\deg(A)<\infty$ then $B:=A\ppp{\deg(A)}$ is non-zero and we have in $B\MModcat{K}$
\begin{equation}
\label{eq:nat-1}%
F_B(A)\simeq \unit\ptimes{ \deg(A)}\,.
\end{equation}
\smallbreak
\item
\label{it:1-3}%
If a tt-functor $F:\cat K\to \cat L$ is weakly conservative on~$\cat K_{\supp(A)}$, see Thm.\,\ref{thm:nat}\,\eqref{it:nat-2}, \eg\ if $F$ is conservative, and if $F(A)\simeq \unit\ptimes{ d}$ in~$\cat L$ then $\deg(A)=d$.
\smallbreak
\item
\label{it:1-4}%
Let $B$ be a tt-ring such that $F_B(A)\simeq \unit\ptimes{ d}$ as tt-rings in $B\MModcat{K}$. Suppose either that $\supp(B)\supseteq\supp(A)$, or that $\cat K$ is local and $B\neq0$. Then $d=\deg(A)$.
\end{enumerate}
\end{Thm}

We need another additive lemma, whose naive proof (with a permutation) fails.
\begin{Lem}
\label{lem:1n'}%
Let $A=\unit\ptimes{n}$. Then there exists an isomorphism $h:F_A(A)=A\otimes A\isoto A\times A\ptimes{(n-1)}$ of $A$-algebras such that $\pr_1h=\mu$.
\end{Lem}

\begin{proof}
To keep track of the various copies of~$\unit$, write $A=\oplus_{i=1}^n\unit_i$ and $A\ptimes{(n-1)}=\oplus_{\ell=1}^{n-1}\oplus_{i=1}^n\unit_{i\ell}$ where $\unit_i=\unit_{i\ell}=\unit$ for all~$i,\ell$.
Then $A\otimes A=\oplus_{i,j}\unit_{i}\otimes \unit_j$. Define $h$ by mapping the summand $\unit_i\otimes \unit_i=\unit$ identically to $\unit_{i}\hook A$ and $\unit_i\otimes \unit_j=\unit$ identically to $\unit_{ij}\hook A\ptimes{(n-1)}$ when $i\neq j$ and $j\leq n-1$ but mapping $\unit_i\otimes \unit_n=\unit$ \emph{diagonally} to $\oplus_{\ell=1,\,\ell\neq i}^{n-1}\unit_{i\ell}\hook A\ptimes{(n-1)}$ for all $i<n$. Verifications are now an exercise.
\end{proof}

\begin{proof}[Proof of Theorem~\ref{thm:1}]
We prove~\eqref{it:1-0} by induction on~$n$. The result is clear for $n=1$. If $A=\unit\ptimes{n}$ for $n\geq 2$ then Lemma~\ref{lem:1n'} gives $A\ppp{2}\simeq A\ptimes{(n-1)}\simeq\unit_{A}\ptimes{(n-1)}$ in $A\MModcat{K}$. By induction hypothesis applied to the tt-category $A\MModcat{K}$ we get $\deg(A\ppp{2})=n-1$ hence the result by definition of the degree. For~\eqref{it:1-1}, we need to prove that there are $A\ppp{n}$-algebra isomorphisms $A\ppp{n}\otimes A\simeq A\ppp{n}\times \cdots \times A\ppp{n}\times A\ppp{n+1}$ (with $n$ factors~$A\ppp{n}$). This is an easy induction on~$n$, applying $A\ppp{n+1}\otimes_{A\ppp{n}}-$ and using~\eqref{eq:An} at each stage. Formula~\eqref{eq:nat-1} follows since $A\ppp{\deg(A)+1}=0$. Parts~\eqref{it:1-3} and~\eqref{it:1-4} follow from~\eqref{it:1-0} and Theorem~\ref{thm:nat}\,\eqref{it:nat-2}-\eqref{it:nat-4}.
\end{proof}

\begin{Cor}
\label{cor:AB}%
Suppose that $\cat K$ is local and that $A,B\in\cat K$ are two tt-rings of finite degree. Then $A\times B$ and $A\otimes B$ have finite degree with $\deg(A\times B)=\deg(A)+\deg(B)$ and $\deg(A\otimes B)=\deg(A)\cdot\deg(B)$.
\end{Cor}

\begin{proof}
By Thm.\,\ref{thm:1}\,\eqref{it:1-2}, there exists two tt-rings $\bar A\neq0$ and $\bar B\neq 0$ such that $F_{\bar A}(A)\simeq \unit\ptimes{\deg(A)}$ and $F_{\bar B}(B)\simeq\unit\ptimes{\deg(B)}$. Let then $\bar C=\bar A\otimes\bar B$. Extending scalars from~$\bar A$ and from~$\bar B$ to~$\bar C$ gives $F_{\bar C}(A\times B)\simeq\unit\ptimes{(\deg(A)+\deg(B))}$ and  $F_{\bar C}(A\otimes B)\simeq\unit\ptimes{(\deg(A)\cdot\deg(B))}$. Finally $\bar C\neq0$ since $\cat K$ is local and apply Theorem~\ref{thm:1}\,\eqref{it:1-4}.
\end{proof}

\begin{Rem}
\label{rem:non-t}%
It will be clear to the interested reader that several arguments, mostly the early ones of Section~\ref{se:split}, only depend on the property that split epimorphisms in~$\cat K$ admit a kernel (a property which holds when $\cat K$ is triangulated, regardless of idempotent-completeness). The reader interested in using the degree in that generality will easily adapt our definition. However, all results which involve $\SpcK$, or the support $\supp(A)$, or the local categories $\cat K/\cat P$, as well as the geometric applications in~\cite{Balmer13ppb}, only make sense when $\cat K$ is triangulated. It is nonetheless interesting to be able to speak of the degree in the generality of, say, the category of abelian groups, for instance.
\end{Rem}

%------------------------------------------------------------------------------
\goodbreak
\section{Examples}
\label{se:exas}%
\medskip
%------------------------------------------------------------------------------

We start by quickly discussing tt-rings of minimal degree (beyond $\deg(0)=0$).

\begin{Prop}
\label{prop:deg=1}%
Let $A$ be a tt-ring with $\deg(A)=1$, \ie such that $\mu:A\otimes A\to A$ is an isomorphism. Then $A\otimes-:\cat K\to \cat K$ is a (very special) Bousfield localization with $F_A:\cat K\to \AK$ as (Verdier) localization. Also, $\Spc(\AK)$ is homeomorphic to the open and closed subset $\supp(A)$ of~$\SpcK$. If $\cat K$ is rigid, this further implies a decomposition $\cat K=\cat K_1\times \cat K_2$ under which $A\cong(\unit,0)$.
\end{Prop}

\begin{proof}
Since $\mu$ is an isomorphism, then so are its two right inverses $\eta\otimes1$ and $1\otimes\eta:A\to A\potimes{2}$ and these inverses coincide. So, $L:=A\otimes-$ is a Bousfield localization ($\eta L=L\eta$ is an isomorphism). Let $C\to \unit\oto{\eta} A\to\Sigma(C)$ be an exact triangle on~$\eta$. Since, $A\otimes\eta$ is an isomorphism, we have $A\otimes C=0$. Therefore $\SpcK=\supp(A)\sqcup\supp(C)$ hence $\supp(A)$ is open and closed. Since every object $x\in\cat K$ fits in an exact triangle $C\otimes x\to x\to A\otimes x\to\Sigma(C\otimes x)$, it is standard to show that the kernel of $A\otimes-$ is exactly the thick $\otimes$-ideal $\cat J:=\ideal{C}$ generated by~$C$ and that $F_A$ induces an equivalence $\cat K/\cat J\isotoo \AK$. Hence $\Spc(F_A)$ induces a homeomorphism $\Spc(\AK)\cong\SET{\cat P\in\SpcK}{\cat J\subseteq \cat P}=\SET{\cat P}{C\in\cat P}=\cat U(C)=\supp(A)$. When $\cat K$ is rigid, $\supp(A)\cap \supp(C)=\varnothing$ forces furthermore $\Homcat{K}(A,\Sigma C)=0$, in which case the above triangle splits\,: $\unit\simeq A\oplus C$. This gives the wanted decomposition, where $\cat K_1=A\otimes \cat K$ and $\cat K_2=C\otimes \cat K$.
\end{proof}

We want to show that the degree is finite in examples. Our main tool is\,:
\begin{Thm}
\label{thm:exas}%
Suppose that $\cat K$ admits a conservative tt-functor $F:\cat K\to \cat L$ into a tt-category $\cat L$ such that every object of $\cat L$ is isomorphic to a sum of suspensions of~$\unitcat{L}$. {\rm(\footnote{Such an $\cat L$ is sometimes called a ``field" but the author finds this definition too restrictive. Also note that the existence of such a functor~$F$ forces~$\cat K$ to be local.})} Then every tt-ring in~$\cat K$ has finite degree. More precisely, if $F(A)\simeq\oplus_{i=k}^\ell\Sigma^{i}\unit^{r_i}$ for $r_k,\ldots,r_\ell\in\bbN$ then $\deg(A)=\sum_{i=k}^\ell r_i$.
\end{Thm}

\begin{proof}
By Theorem~\ref{thm:nat}\,\eqref{it:nat-2}, it suffices to prove that every tt-ring $A$ in $\cat L$ has finite degree $d=\sum_{i=k}^\ell r_i$, where $A\simeq\oplus_{i=k}^\ell\Sigma^{i}\unit^{r_i}$ as objects in $\cat L$. First, let $B=A\ppp{d+1}$. Then, by Theorem~\ref{thm:1}\,\eqref{it:1-1}, we have $F_B(A)\simeq \unit^{d+1}\oplus x$ in $B\MModcat{L}$. On the other hand $F_B(A)\simeq \oplus_{i=k}^\ell\Sigma^{i}\unit^{r_i}$. Therefore there is a split monomorphism $\unit^{d+1}\into \oplus_{i=k}^\ell\Sigma^{i}\unit^{r_i}$ in $B\MModcat{L}$ which can be described by a split injective $d\times (d+1)$-matrix with coefficients in the graded-commutative ring~$S^\sbull=\End^\sbull_B(\unit_B)$. This is impossible (by mapping to a graded residue field of $S^\sbull$) unless $S^\sbull=0$, that is $B=0$ meaning $A\ppp{d+1}=0$. Hence $\deg(A)\leq d$. Now, replace $B$ by $A\ppp{\deg(A)}$ and reason as above. We now have isomorphisms $F_B(A)\simeq \unit^{\deg(A)}$ and $F_B(A)\simeq \oplus_{i=k}^\ell\Sigma^{i}\unit^{r_i}$ in $B\MModcat{L}$ with $B\neq0$. The isomorphism $\unit^{\deg(A)}\simeq \oplus_{i=k}^\ell\Sigma^{i}\unit^{r_i}$ forces (periodicities $\Sigma^{i}\unit\simeq \unit$ in $B\MModcat{K}$ whenever $r_i\neq0$ and) $\deg(A)=\sum_{i=k}^\ell r_i$.
\end{proof}

\begin{Cor}
\label{cor:nat-X}%
Let $X$ be a quasi-compact and quasi-separated scheme (\eg\ an affine or a noetherian scheme). Then every tt-ring in $\Dperf(X)$ has finite degree.
\end{Cor}

\begin{proof}
By Theorem~\ref{thm:local}, we can assume that $X=\Spec(R)$ with $(R,\gm)$ local. Then, the functor $\Dperf(R)\to\Dperf(k)$ to the residue field $k=R/\gm$ is conservative.
\end{proof}

\begin{Exa}
Let $A$ be a separable commutative $R$-algebra which is projective as $R$-module (and finitely generated by~\cite[Prop.\,II.2.1]{DeMeyerIngraham71}). Since $A$ is $R$-flat, we can view it as the ``same" tt-ring in~$\Dperf(R)$. Then its degree can be computed in every residue field, hence $\deg(A)$ coincides with the rank of~$A$ as $R$-module.
\end{Exa}

\begin{Cor}
Let $\bbH$ be a finite dimensional cocommutative Hopf algebra over a field~$\kk$. Then every tt-ring in the bounded derived category $\Db(\bbH\mmod)$ of finitely generated $\bbH$-modules (with $\otimes=\otimes_\kk$) has finite degree.
\end{Cor}

\begin{proof}
Apply Theorem~\ref{thm:exas} to the fiber functor $\Db(\bbH\mmod)\to \Db(\kk)$.
\end{proof}

\begin{Exa}
\label{exa:nat-G/H}%
For any finite group~$G$, all tt-rings in $\DbG$ have finite degree. For every subgroup $H\leq G$, the tt-ring $A=\kk(G/H)$ has finite degree $\deg(A)=\dim_\kk(A)=[G\!:\!H]$ in $\DbG$. Hence $A$ has also finite degree in $\stab(\kk G)\cong\frac{\DbG}{\Dperf(\kk G)}$ by Theorem~\ref{thm:nat}\,\eqref{it:nat-1}. However, if $H<G$ is a strongly $p$-embedded subgroup then $F_A\cong\Res^G_H$ is an equivalence $\stab(\kk G)\isotoo \stab(\kk H)$ and $\eta_A:\unit\isoto A$ is an isomorphism, hence $\deg(A)=1$ in $\stab(\kk G)$. (Example\,: $p=2$ and $C_2<S_3$.)
\end{Exa}

\begin{Exa}
Let $H_1,H_2$ be two non-conjugate cyclic subgroups of order~$p$ in~$G$ (for instance two non-conjugate symmetries in $D_8$ for $p=2$) and consider $A_i=\kk(G/H_i)$ in~$\cat K=\stab(\kk G)$ as above. Then, by the Mackey formula $A_1\otimes A_2=0$. Consequently they have disjoint support and therefore both formulas of Corollary~\ref{cor:AB} fail in this case, showing the importance of our assumption that the category be local. Yet, one can still deduce global formulas via Theorem~\ref{thm:local}.
\end{Exa}

\begin{Cor}
\label{cor:SH}%
In the stable homotopy category $\cat K=\SHfin$ of finite (topological) spectra, every tt-ring has finite degree.
\end{Cor}

\begin{proof}
First note that the result is true in the localizations $\SHfin_{\bbQ}\cong\Db(\bbQ\mmod)$ and $\SHfin_{p}$, at zero and at each prime~$p$. For the latter, it suffices to apply Theorem~\ref{thm:exas} to homology with coefficients in $\bbZ/p$, which is conservative on~$\SHfin_{p}$ and takes values in $\Db(\bbZ/p\mmod)$. Now, if $A$ is a tt-ring in~$\SHfin$, then there exists $m\geq 1$ such that $A\ppp{m}$ goes to zero in~$\SHfin_{\bbQ}$ (since its degree is finite there). Replacing $A$ by $A\ppp{m}$, we can assume that $A$ itself maps to zero in~$\SHfin_{\bbQ}$, that is, $A$ is torsion. But then $A$ is non-zero in $\SHfin_{p}$ for only finitely many primes~$p$. Therefore we can find $n$ big enough so that $A\ppp{n}=0$ everywhere. Hence $A\ppp{n}=0$.
\end{proof}

%------------------------------------------------------------------------------
\medbreak\noindent\textbf{Acknowledgments}\,:
It is my pleasure to thank an anonymous referee for several useful suggestions, including the one discussed in Remark~\ref{rem:non-t} and a simplification of the proof of Corollary~\ref{cor:SH}.
%------------------------------------------------------------------------------

\medbreak
\bibliographystyle{alpha}%
%\bibliography{TG-articles}

%------------------------------------------------------------------------------
%------------------------------------------------------------------------------
\end{document}